\documentclass[11pt]{article}
\usepackage{amsfonts}
\usepackage{amsmath,amssymb}
\usepackage{amsthm}
\usepackage[mathscr]{euscript}
\usepackage{mathrsfs}
\usepackage{indentfirst}
\usepackage{color}
\usepackage{accents}
\usepackage{enumerate}

\newtheorem{theorem}{Theorem}[section]
\newtheorem{lemma}{Lemma}[section]

\newtheorem{remark}{Remark}[section]
\newtheorem{proposition}{Proposition}[section]
\numberwithin{equation}{section}

\newcommand{\lbl}[1]{\label{#1}}
\allowdisplaybreaks

\newcommand{\be}{\begin{equation}}
\newcommand{\ee}{\end{equation}}
\newcommand\bes{\begin{eqnarray}} \newcommand\ees{\end{eqnarray}}
\newcommand{\bess}{\begin{eqnarray*}}
\newcommand{\eess}{\end{eqnarray*}}
\newcommand{\bbbb}{\left\{\begin{aligned}}
\newcommand{\nnnn}{\end{aligned}\right.}
\newcommand{\bea}{\begin{align*}}
\newcommand{\eea}{\end{align*}}

\newcommand\ep{\varepsilon}

\newcommand\kk{\left}
\newcommand\rr{\right}
\newcommand\dd{\displaystyle}

\newcommand\dx{{\rm d}x}
\newcommand\dy{{\rm d}y}

\newcommand\yy{\infty}

\newcommand\ol{\overline}
\newcommand\ud{\underline}

\begin{document}\thispagestyle{empty}
\setlength{\baselineskip}{16pt}
\begin{center}
 {\LARGE\bf Longtime behaviors of an epidemic model}\\[2mm]
 {\LARGE\bf with nonlocal diffusions and a free boundary: spreading-vanishing dichotomy\footnote{This work was supported by NSFC Grant 12301247}}\\[4mm]
 {\Large Xueping Li}\\[0.5mm]
{School of Mathematics and Information Science, Zhengzhou University of Light Industry, Zhengzhou, 450002, China}\\[2.5mm]
{\Large Lei Li\footnote{Corresponding author. {\sl E-mail}: 15838059810@163.com}}\\[0.5mm]
{College of Science, Henan University of Technology, Zhengzhou, 450001, China}
\end{center}

\date{\today}

\begin{quote}
\noindent{\bf Abstract.}We propose a nonlocal epidemic model whose spatial domain evolves over time and is represented by $[0,h(t)]$ with $h(t)$ standing for the spreading front of epidemic. It is assumed that the agents can cross the fixed boundary $x=0$, but they will die immediately if they do it, which implies that the area $(-\yy,0)$ is a hostile environment for the agents. We first show that this model is well posed, then prove that the longtime behaviors are governed by a spreading-vanishing dichotomy and finally give some criteria determining spreading and vanishing. Particularly, we obtain the asymptotical behaviors of the principal eigenvalue of a cooperative system with nonlocal diffusions without assuming the related nonlocal operator is self-adjoint, and the steady state problem of such cooperative system on half space $[0,\yy)$ is studied in detail.


\end{quote}

\section{Introduction}\pagestyle{myheadings}
\renewcommand{\thethm}{\Alph{thm}}
{\setlength\arraycolsep{2pt}
We consider the following epidemic model with nonlocal diffusions and one free boundary
  \bes\left\{\begin{aligned}\label{1.1}
&u_t=d_1\int_0^{h(t)}\!J_1(x-y)u(t,y)\dy-d_1u-au+H(v), \hspace{2mm} t>0, ~ x\in[0,h(t)),\\
&v_t=d_2\int_0^{h(t)}\!J_2(x-y)v(t,y)\dy-d_2v-bv+G(u), \hspace{3mm} t>0, ~ x\in[0,h(t)),\\
&u(t,h(t))=v(t,h(t))=0, \hspace{53mm} t>0,\\
&h'(t)=\int_0^{h(t)}\!\int_{h(t)}^{\yy}\big[\mu_1 J_1(x-y)u(t,x)
+\mu_2 J_2(x-y)v(t,x)\big]\dy\dx, \;\; t>0,\\
&h(0)=h_0>0,~ u(0,x)=u_0(x), ~ v(0,x)=v_0(x),\; ~ x\in[0,h_0],\\
 \end{aligned}\right.
 \ees
 where all parameters are positive constants. The kernel functions $J_1$ and $J_2$ satisfies
 \begin{enumerate}
\item[{\bf(J)}]$J\in C(\mathbb{R})\cap L^{\yy}(\mathbb{R})$, $J(x)\ge0$, $J(0)>0$, $J$ is even, $\dd\int_{\mathbb{R}}J(x)\dx=1$.
 \end{enumerate}
The initial value functions $u_0$ and $v_0$ are in $C([0,h_0])$, positive in $[0,h_0)$ and vanish at $x=h_0$.
 The nonlinear reaction terms $H$ and $G$ satisfy
 \begin{enumerate}
\item[{\bf(H)}]\; $H,G\in C^2([0,\yy))$, $H(0)=G(0)=0$, $H'(z),G'(z)>0$ in $[0,\yy)$, $H''(z), G''(z)<0$ in $(0,\yy)$, and $G(H(\hat z)/a)<b\hat{z}$ for some $\hat{z}>0$.
 \end{enumerate}

 The motivation for proposing \eqref{1.1} comes from the existing works \cite{HY,DL,CDLL,LLW}. To model the spread of an oral-faecal transmitted epidemic, Hsu and Yang \cite{HY} proposed the following PDE system
\bes\label{1.2}
\left\{\!\begin{aligned}
&u_{t}=d_1\Delta u-au+H(v), & &t>0,~x\in\mathbb{R},\\
&v_{t}=d_2\Delta v-bv+G(u), & &t>0,~x\in\mathbb{R},
\end{aligned}\right.
 \ees
where all parameters are positive, and $H(v)$ and $G(u)$ satisfy {\bf (H)}. In models \eqref{1.1} and \eqref{1.2}, $u(t,x)$ and $v(t,x)$ stand for the spatial
concentrations of the bacteria and the infective human population, respectively, at time $t$ and location $x$ in the one dimensional habitat; $-au$ represents the natural death rate of the bacterial population and $H(v)$ denotes the contribution of the infective human to the growth rate of the bacteria; $-bv$ is the fatality rate of the infective human population and $G(u)$ is the infection rate of human population; $d_1$ and $d_2$, respectively, stand for the diffusion rate of bacteria and infective human.

Define
  \bess
  \mathcal{R}_0=\frac{H'(0)G'(0)}{ab}.
 \eess
It was proved in \cite{HY} that when $\mathcal{R}_0>1$, there exists a $c_*>0$ such that \eqref{1.2} has a positive monotone travelling wave solution, which is unique up to translation, if and only if $c\ge c_*$. Moreover, the dynamics of the corresponding ODE system with positive initial value is governed by $\mathcal{R}_0$. More precisely, when $\mathcal{R}_0<1$, $(0,0)$ is globally asymptotically stable; while when $\mathcal{R}_0>1$, there exists a unique positive equilibrium $(u^*,v^*)$ which is uniquely given by
  \bes\label{1.3}
  au^*=H(v^*), ~ ~ bv^*=G(u^*),\ees
and is globally asymptotically stable.

Especially, for model \eqref{1.2} with nonnegative and non-trivial initial value functions, by the maximum principle, we know that the unique solution $(u,v)$ is positive for $t>0$ and $x\in\mathbb{R}$, which implies that the epidemic can spread instantly throughout the whole space. It doesn't seem reasonable. One way to overcome this shortcoming of model \eqref{1.2} is to consider the epidemic model on the domain whose boundary is unknown and varies over time, instead of the fixed boundary domain or the whole space. Du and Lin \cite{DL} incorporated the Stefan boundary condition into the model arising from ecology and proposed the following problem
 \bes\label{1.4}
 \left\{\!\begin{aligned}
&u_{t}=d\Delta u+u(a-bu), & &t>0,~x\in(g(t),h(t)),\\
&u(t,g(t))=u(t,h(t))=0, & &t>0,\\
&g'(t)=-\mu u_x(t,g(t)),~ ~ h'(t)=-\mu u_x(t,h(t)), & & t>0,\\
&-g(0)=h(0)=h_0>0, ~ u(0,x)=\hat{u}_0(x), & & x\in[-h_0,h_0].
\end{aligned}\right.
 \ees
According to \eqref{1.4}, the spatial domain $(g(t),h(t))$ is unknown and evolves over time. The boundary conditions $g'(t)=-\mu u_x(t,g(t))$ and $h'(t)=-\mu u_x(t,h(t))$ are usually referred to as the Stefan boundary condition. Du and Lin found that the dynamics of \eqref{1.4} is governed by a spreading-vanishing dichotomy, a new spreading phenomena resulting from reaction-diffusion model. Inspired by this work, in our series of work \cite{LL,LLW1,LLW2}, we introduced the Stefan boundary condition into \eqref{1.2}, and obtained a spreading-vanishing dichotomy, some criteria for spreading and vanishing, as well as the sharp spreading profiles, which are expected to be proved for model \eqref{1.1} in our future succession work.

As above, the dispersals in models \eqref{1.2} and \eqref{1.4} are approximated by random diffusion $\Delta u$, also known as local diffusion. Recently, it has been increasingly recognized that nonlocal diffusion may be better to describe the spatial movements since it can capture local as well as long-distance dispersal. A commonly used nonlocal diffusion operator takes the form of
\bes\label{1.5}
d\int_{\mathbb{R}^N}J(|x-y|)u(t,y)\dy-du,\ees
where $J$ is the kernel function and $d$ is the diffusion coefficient. The biological interpretation of \eqref{1.5} and its properties can be seen from, for example, \cite{AMRT,KLS,BCV,ZhangW24}.
Using operator \eqref{1.5} or its variation to model the spreading phenomenon from ecology and epidemiology has attracted much attention, and many related works have emerged over past decades. For example, please refer to \cite{Ya, Gar, AC, XLR}. Cao et al \cite{CDLL} and Cort{\'a}zar et al \cite{CQW} independently replaced random diffusion $\Delta u$ in \eqref{1.4} with nonlocal diffusion operator \eqref{1.5}, and considered the following problem
\bes\left\{\begin{aligned}\label{1.6}
&u_t=d\int_{g(t)}^{h(t)}\!\!J(x-y)u(t,y)\dy-du+f(u), & &t>0,~x\in(g(t),h(t)),\\[1mm]
&u(t,x)=0,& &t>0, ~ x\notin(g(t),h(t))\\
&h'(t)=\mu\int_{g(t)}^{h(t)}\!\!\int_{h(t)}^{\infty}
J(x-y)u(t,x)\dy\dx,& &t>0,\\[1mm]
&g'(t)=-\mu\int_{g(t)}^{h(t)}\!\!\int_{-\infty}^{g(t)}
J(x-y)u(t,x)\dy\dx,& &t>0,\\[1mm]
&h(0)=-g(0)=h_0>0,\;\; u(0,x)=u_0(x),& &|x|\le h_0,
 \end{aligned}\right.
 \ees
 where kernel $J$ satisfies {\bf (J)} and $u_0\in C([-h_0,h_0]$, $u_0(\pm h_0)=0<u_0(x)$ in $(-h_0,h_0)$.
The nonlinear term $f$ is of the Fisher-KPP type in \cite{CDLL} and $f\equiv0$ in \cite{CQW}.  The authors in \cite{CDLL} showed that similar to \eqref{1.4}, the dynamics of \eqref{1.6} is also governed by a spreading-vanishing dichotomy. However, when spreading occurs, it was proved in \cite{DLZ} that the spreading speed of \eqref{1.6} is finite if and only if $\int_0^{\yy}xJ(x)\dx<\yy$, which is much different from \eqref{1.4} since the spreading speed of \eqref{1.4} is always finite (see e.g. \cite{DL,DMZ,DLou}).

Later on, the following variant of \eqref{1.6} was proposed by Li et al \cite{LLW}
\bes\label{1.7}
\left\{\begin{aligned}
&u_t=d\dd\int_0^{h(t)}J(x-y)u(t,y)\dy-du+f(u), && t>0,~0\le x<h(t),\\
&u(t,h(t))=0, && t>0,\\
&h'(t)=\mu\dd\int_0^{h(t)}\int_{h(t)}^{\infty}
J(x-y)u(t,x)\dy\dx, && t>0,\\
&h(0)=h_0,\;\; u(0,x)=u_0(x), &&0\le x\le h_0,
\end{aligned}\right.
 \ees
 where $u_0\in C([0,h_0])$ and $u_0(h_0)=0<u_0(x)$ in $[0,h_0)$. This model is derived from the assumption that when the species jumps to the area $(-\yy,0)$, they will die immediately,  which implies that the area $(-\yy,0)$ is a hostile environment for the species. Thus the fixed boundary $x=0$ is subject to a nonlocal analog of the usual homogeneous Dirichlet boundary condition. For model \eqref{1.7}, the spreading-vanishing dichotomy and the estimate for spreading speed were derived in \cite{LLW}.

 Motivated by the above works, we consider \eqref{1.1} and aim at knowing as much as possible about its dynamics. This paper is the first part of our investigation for \eqref{1.1}, and involves the well-posedness, a spreading-vanishing dichotomy and some criteria for spreading and vanishing. Before stating our main results, we would like mention some recent work, which is a small sample of free boundary problems in ecology and epidemiology. For the ones with random diffusion, please refer to \cite{DLou} for monostable, bistable and combustion nonlinear term, \cite{WND} for the epidemic model, \cite{DWu,WQW} for the L-V competition model, \cite{DDL} for the case in the periodic setting, \cite{ZZh} for the blow-up result, \cite{DMZ} for the sharp asymptotical spreading profiles and \cite{Du22} for a detailed survey. For the ones with nonlocal diffusion, please refer to a series of works of Du and Ni \cite{DN2,DN3,DN4,DN5} for spreading speed in homogeneous environment, \cite{ZLZ,PLL,ZhangW23} for the case in periodic environment, \cite{DWZ22, LLW24, LW24} for the competition, prey-predator and mutualist models and \cite{ZZLD,ZLD,DLNZ,WD1,WD2,DNW,NV} for the epidemic model.

The first main result involves the well-posedness of \eqref{1.1}.

\begin{theorem}[Global existence and uniqueness]\label{t1.1}Problem \eqref{1.1} has a unique global solution $(u,v,h)$. Moreover, $(u,v)\in [C([0,\yy)\times[0,h(t)])]^2$, $h\in C^1([0,\yy))$, $h'(t)>0$, $0< u(t,x)\le M_1$ and $0< v(t,x)\le M_2$ in $[0,\yy)\times[0,h(t))$ with some $M_1,M_2>0$.
\end{theorem}

Next we show that the longtime behaviors of \eqref{1.1} are governed by a spreading-vanishing dichotomy.

\begin{theorem}[Spreading-vanishing dichotomy]\label{t1.2} Let $(u,v,h)$ be the unique solution of \eqref{1.1}. Then one of the following alternatives must happen.\vspace{-1.5mm}
\begin{enumerate}[$(1)$]
\item \underline{Spreading:} necessarily $\mathcal{R}_0>1$, $ h_\yy:=\lim_{t\to\yy}h(t)=\yy$,
\[\lim_{t\to\yy}(u(t,x),v(t,x))=(U(x),V(x)) ~ ~ {\rm ~  in ~ }C_{\rm loc}([0,\yy)),\]
 where $(U(x),V(x))$ is the unique bounded positive solution of  \eqref{2.4}.\vspace{-1.5mm}
\item \underline{Vanishing:} $h_{\yy}<\yy$, $\lambda_p(h_{\yy})\le0$ and $\lim_{t\to\yy}\|u(t,\cdot)+v(t,\cdot)\|_{C([0,h(t)])}=0$, where $\lambda_p(h_{\yy})$ is the principal eigenvalue of \eqref{2.1}.
\end{enumerate}
\end{theorem}

Our final result gives a relatively complete criteria for spreading and vanishing.

\begin{theorem}[Criteria for spreading and vanishing]\label{t1.3} Let $(u,v,h)$ be the unique solution of \eqref{1.1}. Then the following results hold.\vspace{-1.5mm}
\begin{enumerate}[$(1)$]
\item If $\mathcal{R}_0\le1$, then vanishing happens.\vspace{-1.5mm}
\item If $\mathcal{R}_*\ge1$, i.e.,
\[\mathcal{R}_*:=\frac{H'(0)G'(0)}{(a+d_1)(b+d_2)}\ge1,\]
then spreading occurs.
  \vspace{-1.5mm}
\item Assume $\mathcal{R}_*<1<\mathcal{R}_0$ and fix all parameters but except for $h_0$ and $\mu_i$ for $i=1,2$. We can find a unique $\ell^*>0$ such that\vspace{-1.5mm}
\begin{enumerate}
\item[{\rm (3a)}] if $h_0\ge\ell^*$, then spreading happens;
\item[{\rm (3b)}] if $h_0<\ell^*$, then the following statements hold:
\begin{enumerate}\vspace{-1.5mm}
\item[{\rm (3b$_1$)}] there exists $\underline{\mu}>0$ such that vanishing happens when $\mu_1+\mu_2\le\underline{\mu}$; and there exists a $\bar{\mu}_1>0$ $(\bar{\mu}_2>0)$ which is independent of $\mu_2$\,$(\mu_1)$ such that spreading happens when  $\mu_1\ge\bar{\mu}_1$\,$(\mu_2\ge\bar{\mu}_2)$;
\item[{\rm (3b$_2$)}] if $\mu_2=f(\mu_1)$ where $f\in C([0,\yy))$, is strictly increasing, $f(0)=0$ and $\lim_{s\to\yy}f(s)=\yy$, then there exists a unique $\mu^*_1>0$ such that spreading happens if and only if $\mu_1>\mu^*_1$. \vspace{-1.5mm}
    \end{enumerate}
    \end{enumerate}
    \end{enumerate}
\end{theorem}

\begin{remark}\label{r1.1} When spreading occurs, the spreading speed of \eqref{1.1} will be discussed in a separate work. It is expected that the accelerated spreading will happen if kernel functions $J_1$ and $J_2$ violate a threshold condition. Moreover, for some special kernel functions, such as algebraic decay kernels, it will be proved that the rate of accelerated spreading is closely related to the behaviors of kernel functions $J_1$ and $J_2$ near infinity.

Moreover, we are very interested in the asymptotical behaviors of solution component $(u,v)$ of \eqref{1.1} near the spreading front $h(t)$, which will bring some deep insights for the propagation of epidemic. The related results for \eqref{1.4} or L-V competition system with local diffusions have been obtained in \cite{DMZ} and \cite{WQW}, respectively. However, for models \eqref{1.1}, \eqref{1.6} and \eqref{1.7}, until now there is no work to provide some results.
\end{remark}

This paper is arranged as follows. Section 2 mainly involves a principal eigenvalue problem of a cooperative system with nonlocal diffusions and a steady state problem on half space. Particularly, in contrast to \cite{NV}, we give the exact asymptotical behaviors of principal eigenvalue without assuming that the nonlocal operator $\mathcal{L}$ defined in \eqref{2.1} is self-adjoint; we also obtain a rather complete understanding for the unique bounded positive solution of a steady state problem on half space $[0,\yy)$. In Section 3, by using the results derived in Section 2 and some comparison arguments, we study the dynamics of \eqref{1.1} including  a spreading-vanishing and some criteria for spreading and vanishing.

\section{An eigenvalue problem and a steady state problem associated to \eqref{1.1}}\lbl{s2}
{\setlength\arraycolsep{2pt}

This section involves a principle eigenvalue problem and a steady state problem that are crucial for our discussion on the dynamics of \eqref{1.1}.

For any $l_2>l_1$, we consider the following eigenvalue problem
 \bes\label{2.1}
 \mathcal{L}[\phi](x):=\mathcal{P}[\phi](x)+H(x)\phi(x)=\lambda\phi(x), ~ ~x\in[l_1,l_2],\ees
where $\phi=(\phi_1,\phi_2)^T$,
 \[\mathcal{P}[\phi](x)=\begin{pmatrix}
 d_1\dd\int_{l_1}^{l_2}J_1(x-y)\phi_1(y)\dy \\
 d_2\dd\int_{l_1}^{l_2}J_2(x-y)\phi_2(y)\dy
 \end{pmatrix},~ ~ B=\begin{pmatrix}
  -d_1-a & H'(0) \\
   G'(0) & -d_2-b
  \end{pmatrix}.\]}
For clarity, we make the following notations
 \bess
E&=&[L^2([0,l])]^2, ~ \langle\phi,\psi\rangle=\sum_{i=1}^2\int_0^l\phi_i(x)\psi_i(x)\dx, ~  \|\phi\|_2=\sqrt{\langle\phi,\phi\rangle}, ~ X=[C([0,l])]^2,\\
X^+&=&\{\phi\in X: \phi_1\ge0, \phi_2\ge0 ~ {\rm in ~ }[0,l]\}, ~ X^{++}=\{\phi\in X: \phi_1>0, \phi_2>0 ~ {\rm in ~ }[0,l]\}.
 \eess
Let
\[A=\begin{pmatrix}
  -a & H'(0) \\
   G'(0) & -b
    \end{pmatrix}.\]
Direct computations show there exist $\gamma_A, \gamma_B\in\mathbb{R}$, $\theta_A>0$ and $\theta_B>0$ satisfying
\bes\left\{\begin{aligned}\label{2.2}
&\gamma_A=\frac{-a-b+\sqrt{(a-b)^2
+4H'(0)G'(0)}}2,\\[1mm]
&\gamma_B=\frac{-a-d_1-b-d_2
+\sqrt{(a+d_1-b-d_2)^2+4H'(0)G'(0)}}2,\\[1mm] &\theta_A=\frac{H'(0)}{\gamma_A+a},\;\;
\theta_B=\frac{H'(0)}{\gamma_B+d_1+a}, ~ (\gamma_AI-A)(\theta_A,1)^T=0, ~ ~ (\gamma_BI-B)(\theta_B,1)^T=0.
    \end{aligned}\right.\quad
    \ees
    Moreover, $\gamma_A>\max\{-a,-b\}$ and $\gamma_B>\max\{-d_1-a,-d_2-b\}$. Then we define a value
  \[\lambda_p=\inf\{\lambda\in\mathbb{R}: \mathcal{L}[\phi](x)\le\lambda\phi(x) ~ {\rm in } ~ [l_1,l_2] { \rm ~ for ~ some ~ } \phi\in X^{++}\}.\vspace{-1.5mm}\]
  Clearly, $\lambda_p$ is well-defined.
    To discuss the asymptotical behaviors of principal eigenvalue, we need the following two lemmas.
    \begin{lemma}\label{l2.1}Let $\lambda_1$ be an eigenvalue of \eqref{2.1} with its corresponding eigenfunction belonging to $X^{++}$. Then the following statements are valid.\vspace{-1.5mm}
 \begin{enumerate}[$(1)$]
\item If there exist $\phi=(\phi_1,\phi_2)^T\in X$ with $\phi_1,\phi_2\ge,\not\equiv0$ and $\lambda\in\mathbb{R}$ such that $\mathcal{L}[\phi]\le\lambda\phi$, then $\lambda_1\le\lambda$. Moreover, $\lambda_1=\lambda$ only if $\mathcal{L}[\phi]=\lambda\phi$.\vspace{-1.5mm}
\item If there exist $\phi=(\phi_1,\phi_2)^T\in X^+\setminus\{(0,0)\}$ and $\lambda\in\mathbb{R}$ such that $\mathcal{L}[\phi]\ge\lambda\phi$, then $\lambda_1\ge\lambda$. Moreover, $\lambda_1=\lambda$ only if $\mathcal{L}[\phi]=\lambda\phi$.\vspace{-1.5mm}
   \end{enumerate}
   \end{lemma}
   \begin{proof}
     Since this lemma can be proved by the similar ways as in \cite[Lemma 2.2]{DN8}, the details are omitted here.
   \end{proof}

   \begin{lemma}[{\cite[Lemma 2.9]{DN3}}]\label{l2.2}Assume that $l>0$ and kernel function $J$ satisfies {\bf (J)}. Define $\xi(x)=l-|x|$.
   Then for any small $\ep>0$, there exists a $l_{\ep}>0$ depending only on $J$ and $\ep$ such that if $l\ge l_{\ep}$, then
   \[\int_{-l}^{l}J(x-y)\xi(y)\dy\ge(1-\ep)\xi(x) ~ ~ {\rm for ~ }x\in[-l,l].\]
   \end{lemma}

Now we are in the position to study the eigenvalue problem $\mathcal{L}[\phi]=\lambda\phi$.
It is well known that $\lambda$ is a principal eigenvalue if it is simple and its corresponding eigenfunction $\phi$ is in $X^{++}$.

\begin{proposition}\label{p2.1} Let $\mathcal{L}$ be defined as above. Then the following statements are valid.\vspace{-1.5mm}
 \begin{enumerate}[$(1)$]
 \item $\lambda_{p}$ is an eigenvalue of operator $\mathcal{L}$ with a corresponding eigenfunction $\phi_p\in X^{++}$.
 \item The algebraic multiplicity of $\lambda_p$ is equal to one. Namely, $\lambda_p$ is simple.\vspace{-1.5mm}
 \item If there exists an eigenpair $(\lambda,\phi)$ of $\mathcal{L}$ with $\phi\in X^+\setminus\{(0,0)\}$, then $\lambda=\lambda_p$ and $\phi$ is a positive constant multiple of $\phi_p$.\vspace{-1.5mm}
 \item To stress the dependence of $\lambda_p$ on the length of $[l_1,l_2]$, we rewrite $\lambda_p$ as $\lambda_p(l_1,l_2)$. Then $\lambda_p(l_1,l_2)$ is strictly increasing in $l_2-l_1$ and continuous for all $l_2>l_1$. Moreover, $\lambda_p(l_1,l_2)\to\gamma_A$ as $l_2-l_1\to\yy$, and $\lambda_p(l_1,l_2)\to\gamma_B$ as $l_2-l_1\to0$, where $\gamma_A$ and $\gamma_B$ are given by \eqref{2.2}.
 \end{enumerate}
\end{proposition}

\begin{proof}It is easy to verify that the statements (1)-(3) can be directly derived by \cite[Corollary 1.3 and Theorem 1.4]{SWZ}. Thus we only prove the statement (4).

Since the continuity and monotonicity can be obtained by following the analogous methods as in the proof of \cite[Proposition 2.3]{DN8}, the details are omitted. We next prove the limits of $\lambda_p(l_1,l_2)$ which will be done by virtue of Lemmas \ref{l2.1} and \ref{l2.2}. It is not hard to check that $\lambda_p(l_1,l_2)$ depends only on the length $l_2-l_1$ of interval $[l_1,l_2]$. Hence it is sufficient to prove the limits for $\lambda_p(-l,l)$. For simplicity, denote $\lambda_p(-l,l)$ by $\lambda_p(l)$.

Recall that $\gamma_A$ and $\theta_A$ are given by \eqref{2.2}. Define $\bar{\varphi}=(\theta_A,1)^T$. Simple calculations lead to
 \bess\begin{cases}
\dd d_1\int_{-l}^{l}J_1(x-y)\theta_A\dy-d_1\theta_A-a\theta_A+H'(0)\le -a\theta_A+H'(0)=\gamma_A\theta_A,\\
\dd d_2\int_{-l}^{l}J_2(x-y)\dy-d_2+G'(0)\theta_A-b\le G'(0)\theta_A-b=\gamma_A.
\end{cases}
 \eess
 Thus $\mathcal{L}[\bar\varphi]\le \gamma_A\bar{\varphi}$ for all $l>0$ which, combined with Lemma \ref{l2.1}, leads to
  \bes\label{2.3}
  \lambda_p(l)\le\gamma_A\;\;\;\text{for all}\;\,l>0.
 \ees

Define $\underline{\varphi}=(\underline{\varphi}_1(x),\underline\varphi_2(x))^T$ where
\[\underline{\varphi}_1(x)=\theta_A\xi(x), ~  \underline{\varphi}_2(x)=\xi(x) {\rm ~ and ~ } \xi(x)=l-|x|.\]

{\bf Claim.} For any small $\ep>0$ there exists $l_{\ep}>0$ such that when $l>l_{\ep}$ there holds:
 \bess
 \mathcal{L}[\underline{\varphi}]\ge (\gamma_A-\max\{d_1,d_2\}\ep)\underline{\varphi} ~ ~{\rm for ~ }x\in[-l,l].
\eess
Once this claim is verified, by Lemma \ref{l2.1}, we get $\lambda_p(l)\ge \gamma_A-\ep$ for $l\ge l_{\ep}$. Then in view of the arbitrariness of $\ep$, we have $
  \liminf_{l\to\yy}\lambda_p(l)\ge\gamma_A$.

Let us now verify the above claim. In light of Lemma \ref{l2.2}, for any small $\ep>0$, there exists a $l_{\ep}>0$ such that for any $l\ge l_{\ep}$ and $x\in[-l,l]$,
\bess\begin{cases}
\dd d_1\int_{-l}^{l}J_1(x-y)\underline{\varphi}_1(y)\dy-d_1\underline{\varphi}_1-a\underline{\varphi}_1+H'(0)\underline{\varphi}_2\ge(\gamma_A-d_1\ep)\theta_A,\\
\dd d_2\int_{-l}^{l}J_2(x-y)\underline{\varphi}_2(y)\dy-d_2\underline{\varphi}_1+G'(0)\underline{\varphi}_1-b\underline{\varphi}_2\ge(\gamma_A-d_2\ep).
\end{cases}
\eess
So our claim holds. Together with \eqref{2.3}, we derive the first limit.

Recall that $\gamma_B$ and $\theta_B$ are determined in \eqref{2.2}. Let $\underline\psi=(\theta_B,1)^T$. We claim that $\mathcal{L}[\psi]\ge\gamma_B\psi$ for all $l>0$. It is easy to verify
\bess
\begin{cases}
\dd d_1\int_{-l}^{l}J_1(x-y)\dy\theta_B-d_1\theta_B-a\theta_B+H'(0)
\ge -d_1\theta_B-a\theta_B+H'(0)=\gamma_B\theta_B,\\
\dd d_2\int_{-l}^{l}J_2(x-y)\dy-d_2-b+G'(0)\theta_B
\ge -d_1-b+G'(0)\theta_B=\gamma_B.
\end{cases}
\eess
Thus $\mathcal{L}[\ud\psi]\ge\gamma_B\ud\psi$ for all $l>0$. Due to Lemma \ref{l2.1}, we obtain $\lambda_p(l)\ge \gamma_B$ for all $l>0$.

On the other hand, straightforward computations show
 \bess
\dd d_1\int_{-l}^{l}J_1(x-y)\dy\theta_B-d_1\theta_B-a\theta_B+H'(0)
&\le& 2d_1l\|J_1\|_{\yy}\theta_B-d_1\theta_B-a\theta_B+H'(0)\\
&=&(\gamma_B+2d_1l\|J_1\|_{\yy})\theta_B,\\
\dd d_2\int_{-l}^{l}J_2(x-y)\dy-d_2-b+G'(0)\theta_B
&\le& 2d_2l\|J_2\|_{\yy}-d_1-b+G'(0)\theta_B\\
&=&\gamma_B+2d_2l\|J_2\|_{\yy}.
\eess
Using Lemma \ref{l2.1} again, we get $\lambda_p(l)\le \gamma_B+2l\max_{i=1,2}\{d_i\|J_i\|_{\yy}\}$, which implies $\limsup_{l\to0}\lambda_p(l)\le \gamma_B$. Together with $\lambda_p(l)\ge\gamma_B$ for all $l>0$, we finish the proof of the second limit. The proof is complete.
\end{proof}

According to Proposition \ref{p2.1}, we know that for principal eigenvalue $\lambda_p(l)$, there exists a unique eigenfunction $\phi^{l}=(\phi^l_1,\phi^l_2)\in X^{++}$ with $\|\phi^l\|_X=1$ . We now show $\phi^l$ is also continuous for $l>0$ in the following sense, which is independently interesting.

For any $l_0>0$ and a sequence $\{l_n\}$ with $l_n\to l_0$ as $n\to\yy$, we denote the corresponding triplet by $(\lambda_n,\phi^n_1,\phi^n_2)$ with $\phi^{n}=(\phi^n_1,\phi^n_2)\in X^{++}$ and $\|\phi^n\|_X=1$. Define for $x\in\mathbb{R}$,
\bess
\tilde{\phi}^n_1=\frac{\left|\begin{matrix}
                               \dd d_1\int_{-l_n}^{l_n}J_1(x-y)\phi^n_1(y)\dy & -H'(0) \\
                               \dd d_2\int_{-l_n}^{l_n}J_2(x-y)\phi^n_2(y)\dy & d_2+b+\lambda_n
                             \end{matrix}\right|}{\left|\begin{matrix}
                               d_1+a+\lambda_n & -H'(0) \\
                               -G'(0) & d_2+b+\lambda_n
                             \end{matrix}\right|},
\eess
and
\bess
\tilde{\phi}^n_2=\frac{\left|\begin{matrix}
                               d_1+a+\lambda_n & \dd d_1\int_{-l_n}^{l_n}J_1(x-y)\phi^n_1(y)\dy \\
                              -G'(0)  & \dd d_2\int_{-l_n}^{l_n}J_2(x-y)\phi^n_2(y)\dy
                             \end{matrix}\right|}{\left|\begin{matrix}
                               d_1+a+\lambda_n & -H'(0) \\
                               -G'(0) & d_2+b+\lambda_n
                             \end{matrix}\right|}.
\eess
Recall that $\lambda_n,\lambda_p(l_0)\in(\gamma_B,\gamma_A)$. Thus $(\tilde{\phi}^n_1,\tilde{\phi}^n_2)$ is well defined in $\mathbb{R}$.
Moreover, $(\tilde{\phi}^n_1,\tilde{\phi}^n_2)$ satisfies
\bess
\begin{cases}
\dd d_1\int_{-l_n}^{l_n}J_1(x-y)\phi^n_1(y)\dy-d_1\tilde{\phi}^n_1-a\tilde{\phi}^n_1+H'(0)\tilde{\phi}^n_2=\lambda_n\tilde{\phi}^n_1, ~ ~ x\in\mathbb{R},\\
\dd d_2\int_{-l_n}^{l_n}J_2(x-y)\phi^n_2(y)\dy-d_2\tilde{\phi}^n_2-b\tilde{\phi}^n_2+G'(0)\tilde{\phi}^n_1=\lambda_n\tilde{\phi}^n_2, ~ ~ x\in\mathbb{R}.
\end{cases}
\eess
Clearly, $(\phi^n_1,\phi^n_2)=(\tilde{\phi}^n_1,\tilde{\phi}^n_2)$ for $x\in[-l_n,l_n]$. It is not hard to check that $(\tilde{\phi}^n_1,\tilde{\phi}^n_2)$ is equi-continuous and uniformly bounded for $n\ge1$ and $x\in\mathbb{R}$. Then by a compact argument, there exists a subsequence, still denoted by itself, such that $(\tilde{\phi}^n_1,\tilde{\phi}^n_2)\to(\psi_1,\psi_2)$ locally uniformly in $\mathbb{R}$ as $n\to\yy$. Letting $n\to\yy$, we have
\bess
\begin{cases}
\dd d_1\int_{-l_0}^{l_0}J_1(x-y)\psi_1(y)\dy-d_1\psi_1-a\psi_1+H'(0)\psi_2=\lambda_p(l_0)\psi_1, ~ ~ x\in[-l_0,l_0],\\
\dd d_2\int_{-l_0}^{l_0}J_2(x-y)\psi_2(y)\dy-d_2\psi_2-b\psi_2+G'(0)\psi_1=\lambda_p(l_0)\psi_2, ~ ~  x\in[-l_0,l_0].
\end{cases}
\eess
Furthermore, since $\|\phi^n\|_X=1$ (here $X=C([-l_n,l_n])$), $(\phi^n_1,\phi^n_2)=(\tilde{\phi}^n_1,\tilde{\phi}^n_2)$ for $x\in[-l_n,l_n]$, $l_n\to l_0$ and $(\tilde{\phi}^n_1,\tilde{\phi}^n_2)\to(\psi_1,\psi_2)$ locally uniformly in $\mathbb{R}$ as $n\to\yy$, we can show that $\max_{x\in[-l_0,l_0]}(\psi_1(x)+\psi_2(x))=1$. In view of Proposition \ref{p2.1}, we get that $(\psi_1,\psi_2)$ is the unique positive eigenfunction of $\lambda_p(l_0)$ with $\|(\psi_1,\psi_2)\|_X=1$ where $X=C([-l_0,l_0])$.

Now we investigate the following nonlocal steady state problem on the half space
\bes\label{2.4}
\begin{cases}
 \dd d_1\int_{0}^{\yy}J_1(x-y)U(y)\dy-d_1U-aU+H(V)=0, ~ ~ x\in[0,\yy),\\
\dd d_2\int_{0}^{\yy}J_2(x-y)V(y)\dy-d_2V-bV+G(U)=0, ~ ~  x\in[0,\yy).
\end{cases}
\ees

\begin{proposition}\label{p2.2}Assume that $\mathcal{R}_0>1$. Then problem \eqref{2.4} has a unique bounded positive solution $(U,V)$. Moreover, $(U,V)\in [C([0,\yy))]^2$, $(0,0)<(U(x),V(x))<(u^*,v^*)$ for $x\ge0$, $(U(x),V(x))$ is strictly increasing in $x\ge0$ and $(U(x),V(x))\to(u^*,v^*)$ as $x\to\yy$, where $(u^*,v^*)$ is unique given by \eqref{1.3}.
\end{proposition}
\begin{proof}Since this proof is long, it will be divided into three steps.

{\bf Step 1:}{\it The existence.} For clarity, we denote the principal eigenvalue of \eqref{2.1} on interval $[0,l]$ by $\lambda_p(l)$. Due to $\mathcal{R}_0>1$, from Proposition \ref{p2.1} it follows that $\lambda_p(l)>0$ for all large $l>0$. Let $\phi=(\phi_1,\phi_2)$ be the corresponding positive eigenfunction of $\lambda_p(l)$ with $\|\phi\|_X=1$.
 Define an operator ${\it\Gamma}$: $X^+\to X^+$ by \[{\it\Gamma}[\phi]=\begin{pmatrix}
 \dd\frac1{d_1+a}\left(d_1\int_0^lJ_1(x-y)\phi_1(y)\dy+H(\phi_2)\right) \\[4mm] \dd\frac1{d_2+b}\left(d_2\int_0^lJ_2(x-y)\phi_2(y)\dy+G(\phi_1)\right)
  \end{pmatrix}.\]
Clearly, ${\it\Gamma}$ is increasing in $X^+$, i.e., $\Gamma(\phi)\ge\Gamma(\tilde{\phi})$ for any $\phi,\tilde{\phi}\in X^+$ and $\phi-\tilde{\phi}\in X^+$. Simple computations show
 \bess
 \frac1{d_1+a}\left(d_1\int_0^lJ_1(x-y)u^*\dy+H(v^*)\right)&\le&
 \frac1{d_1+a}\left[d_1u^*+H(v^*)\right]\nonumber\\
 &=&\frac1{d_1+a}\left[d_1u^*+au^*\right]\nonumber\\[1mm]
 &=&u^*,\;\;\;x\in[0,l],\\
 \frac1{d_2+b}\left(d_2\int_0^lJ_2(x-y)v^*\dy+G(u^*)\right)&\le&
 v^*,\;\;\;x\in[0,l],
 \eess
which implies that ${\it\Gamma}[(u^*,v^*)]\le(u^*,v^*)$.

We now show that if $\ep$ is sufficiently small, then ${\it\Gamma}[\ep\phi]\ge\ep\phi$. Direct calculations yield
 \bess
 &&\frac1{d_1+a}\left[d_1\int_0^lJ_1(x-y)\ep\phi_1(y)\dy+H(\ep\phi_2)\right]
 -\ep\phi_1\\[1mm]
 &=&\frac{\ep}{d_1+a}\left[\lambda_p(l)\phi_1+d_1\phi_1+a\phi_1-H'(0)\phi_2+\frac{H(\ep\phi_2)}{\ep}\right]
 -\ep\phi_1\\[1mm]
 &=&\frac{\ep}{d_1+a}\left[\lambda_p(l)\phi_1-H'(0)\phi_2+\frac{H(\ep\phi_2)}{\ep}\right]\\[1mm]
 &\ge&\frac\ep{d_1+a}\left[\lambda_1(l)\phi_1+\frac{H(\ep)}{\ep}-H'(0)\right]\\[1mm]
 &\ge&\frac\ep{d_1+a}\left[\lambda_1(l)\min_{x\in[0,l]}\phi_1(x)+o(1)\right]\ge0
 \eess
provided that $\ep$ is small enough. Similarly,
 \[\frac1{d_2+b}\left[d_2\int_0^lJ_2(x-y)\ep\phi_2(y)\dy
 +G(\ep\phi_1)\right]\ge\ep\phi_2\]
 with $\ep$ small enough. Thus ${\it\Gamma}[\ep\phi]\ge\ep\phi$. We further shrink $\ep$ if necessary such that $(\ep\phi_1,\ep\phi_2)\le(u^*,v^*)$ for $x\in[0,l]$.

Then by an iteration or upper-lower solution method, we see that problem
\bes\label{2.5}
\begin{cases}
 \dd d_1\int_{0}^{l}J_1(x-y)u(y)\dy-d_1u-au+H(v)=0, ~ ~ x\in[0,l],\\
\dd d_2\int_{0}^{l}J_2(x-y)v(y)\dy-d_2v-bv+G(u)=0, ~ ~  x\in[0,l]
\end{cases}
\ees
 has at least one  solution $(u_l,v_l)$ satisfying $(\ep\phi_1,\ep\phi_2)\le(u_l,v_l)\le (u^*,v^*)$ in $[0,l]$.

We now prove $(u_l,v_l)$ is continuous in $[0,l]$ by using the implicit function theorem and some basic analysis. Define
 \bess
 &Q_1(x)=\dd d_1\int_0^lJ_1(x-y)u_l(y)\dy, ~ ~ Q_2(x)=d_2\int_0^lJ_2(x-y)v_l(y)\dy,\\[1mm]
 &P(x,y,z)=\big(Q_1(x)-d_1y-ay+H(z), \;Q_2(x)-d_2z-bz+G(y)\big).
 \eess
Clearly, $P(x,y,z)$ is continuous in $\{(x,y,z):0\le x\le l,y\ge0,z\ge0\}$, and $P(x, u_l(x),v_l(x))=(0,0)$ for all $0<x<l$.
With regard to $0<x<l$, $y>0$, $z>0$ satisfying $P(x,y,z)=(0,0)$, in view of {\bf (H)}, direct computations yield
 \bess
 \frac{\partial P(x,y,z)}{\partial(y,z)}&=&\begin{pmatrix}                                                  -d_1-a &  H'(z) \\
  G'(y) & -d_2-b
   \end{pmatrix},
 \eess
which is continuous for $0<x<l$, $y,z>0$, and
 \bess
 {\rm det}\frac{\partial P(x,y,z)}{\partial(y,z)}&=&(d_1+a)(d_2+b)-H'(z)G'(y)\\[1mm]
 &=&\frac{(Q_1(x)+H(z))(Q_2(x)+G(y))}{yz}-H'(z)G'(y)\\[1mm]
 &\ge&\frac{H(z)G(y)}{yz}-H'(z)G'(y)>0.
 \eess
Hence, by the implicit function theorem, we know that $(u_l,v_l)$ is continuous in $(0,l)$.  Then it remains to show the continuity at $0$ and $l$. Since the methods are similar, we only handle the case at $x=0$. Clearly, it suffices to prove that for any sequence $\{x_n\}$ converging to $0$, there exists a subsequence, still denoted by itself, such that $u_l(x_n)\to u_l(0)$ and $v_l(x_n)\to v_l(0)$ as $n\to\yy$. Due to the boundedness of $u_l$ and $v_l$, by passing a subsequence denoted by itself, we have $u_l(x_n)\to u_0$ and $v_l(x_n)\to v_0$ for some $u_0>0$ and $v_0>0$. Notice that the integral terms in \eqref{2.5} are continuous in $[0,l]$. Letting $n\to\yy$, thanks to \eqref{2.5} we have
 \bes\label{2.6}
 \left\{\begin{aligned}
&d_1\int_0^{l}J_1(y)u_l(y)\dy-d_1u_0-au_0+H(v_0)=0,\\
&d_2\int_0^{l}J_2(y)v_l(y)\dy-d_2v_0-bv_0+G(u_0)=0.
 \end{aligned}\right.
 \ees
Moreover, substituting $x=0$ into \eqref{2.5} leads to
  \bes\label{2.7}
 \left\{\begin{aligned}
&d_1\int_0^{l}J_1(y)u_l(y)\dy-d_1u_l(0)-au_l(0)+H(v_l(0))=0,\\
&d_2\int_0^{l}J_2(y)v_l(y)\dy-d_2v_l(0)-bv_l(0)+G(u_l(0))=0.
 \end{aligned}\right.
 \ees
 Define
 \[\vartheta^*_0=\inf\{\vartheta>1: \vartheta u_l(0)\ge u_0, \vartheta v_l(0)\ge v_0\}.\]
Obviously, $\vartheta^*_0$ is well defined and $\vartheta^*_0\ge1$. We next show $\vartheta^*_0=1$. Assume on the contrary that $\vartheta^*>1$. Then by the definition of $\vartheta^*_0$, we must have $\vartheta^*_0u_l(0)=u_0$ or $\vartheta^*_0v_l(0)=v_0$. Without loss of generality, let $\vartheta^*_0u_l(0)=u_0$. Then by the first equalities of \eqref{2.6} and \eqref{2.7}, it is not hard to derive $\vartheta^*_0H(v_l(0))<H(v_0)$. However, by {\bf (H)}, we have
\bess
H(v_0)-\vartheta^*_0H(v_l(0))=\vartheta^*_0v_l(0)\bigg(\frac{H(v_0)}{\vartheta^*_0v_l(0)}-\frac{H(v_l(0))}{v_l(0)}\bigg)
\le\vartheta^*_0v_l(0)\bigg(\frac{H(\vartheta^*_0v_l(0))}{\vartheta^*_0v_l(0)}-\frac{H(v_l(0))}{v_l(0)}\bigg)<0.
\eess
This contradiction implies that $u_l(0)\ge u_0$ and $v_l(0)\ge v_0$. By changing the position of $(u_l(0),v_l(0))$ and $(u_0,v_0)$, we can deduce $u_0\ge u_l(0)$ and $v_0\ge v_l(0)$. So $u_l(0)=u_0$ and $v_l(0)=v_0$. Therefore, $(u_l,v_l)$ is continuous in $[0,l]$.

Now we claim that $(u_l,v_l)$ is increasing for all large $l$, i.e., $(u_{l_1}(x),v_{l_1}(x))\le (u_{l_2}(x),v_{l_2}(x))$ for any large $l_2>l_1$ and $x\in[0,l_1]$. Recall that $(u_l,v_l)$ is continuous and positive in $[0,l]$. Thus we can define
\[k^*=\inf\{k>1: (u_{l_1}(x),v_{l_1}(x))\le k(u_{l_2}(x),v_{l_2}(x)) ~ ~ {\rm for ~ }x\in[0,l_1]\}.\]
Clearly, $k^*\ge1$ and $(u_{l_1}(x),v_{l_1}(x))\le k^*(u_{l_2}(x),v_{l_2}(x))$ for $x\in[0,l_1]$. If $k^*>1$, then by continuity there must exist some $x_0\in[0,l]$ such that $k^*u_{l_2}(x_0)=u_{l_1}(x_0)$ or $k^*v_{l_2}(x_0)=v_{l_1}(x_0)$. Without loss of generality, we suppose $k^*u_{l_2}(x_0)=u_{l_1}(x_0)$. Moreover, $(u_{l_2},v_{l_2})$ satisfies
\bess
\begin{cases}
 \dd d_1\int_{0}^{l_1}J_1(x-y)u_{l_2}(y)\dy-d_1u_{l_2}-au_{l_2}+H(v_{l_2})\le0, ~ ~ x\in[0,l_1],\\
\dd d_2\int_{0}^{l_1}J_2(x-y)v_{l_2}(y)\dy-d_2v_{l_2}-bv_{l_2}+G(u_{l_2})\le0, ~ ~  x\in[0,l_1].
\end{cases}
\eess
Combing with \eqref{2.5} on $[0,l_1]$ and arguing as in the proof of the continuity of $(u_l,v_l)$ on $x=0$, we can obtain a contradiction. Thus $k^*=1$, and our claim is proved.

According to the above analysis, we know $U(x):=\lim_{l\to\yy}u_l(x)$ and $V(x):=\lim_{l\to\yy}v_l(x)$ for all $x\ge0$ are well defined. Moreover, $(0,0)<(U(x),V(x))\le (u^*,v^*)$ for $x\ge0$. By the dominated convergence theorem, we see that $(U,V)$ satisfies \eqref{2.4}. Hence the existence is proved.

{\bf Step 2:} {\it The uniqueness.} Firstly, by the similar methods as in the proof of the continuity of $(u_l,v_l)$, we can show that $(U,V)$ is continuous in $x\ge0$. Now we show that $(U(x),V(x))\to(u^*,v^*)$ as $x\to\yy$. Otherwise, there exist $\ep_0>0$ and a sequence $\{x_n\}$ satisfying $x_n\nearrow\yy$ as $n\to\yy$ such that $U(x_n)\le u^*-\ep_0$ or $V(x_n)\le v^*-\ep_0$. Without loss of generality, we suppose that $U(x_n)\le u^*-\ep_0$ for $n\ge1$.

Set $w_n(x)=u_{2x_n}(x+x_n)$ and $z_n(x)=v_{2x_n}(x+x_n)$. Since $(u_{2x_n},v_{2x_n})$ satisfies \eqref{2.5} on $[0,2x_n]$, we have
\bess
\begin{cases}
 \dd d_1\int_{-x_n}^{x_n}J_1(x-y)w_n(y)\dy-d_1w_n-aw_n+H(z_n)=0, ~ ~ x\in[-x_n,x_n],\\
\dd d_2\int_{-x_n}^{x_n}J_2(x-y)z_n(y)\dy-d_2z_n-bz_n+G(w_n)=0, ~ ~ x\in[-x_n,x_n].
\end{cases}
\eess
Thus $(w_n,z_n)$ is the positive solution of \eqref{2.5} on $[-x_n,x_n]$. From \cite[Proposition 2.11]{NV}, we know $(w_n(x),z_n(x))\to(u^*,v^*)$ locally uniformly in $\mathbb{R}$ as $n\to\yy$. Hence for large $n$,
\[(u_{2x_n}(x_n),v_{2x_n}(x_n))=(w_n(0),z_n(0))\ge(u^*-\frac{\ep_0}{2},v^*-\frac{\ep_0}{2}).\]
By the definition of $(U,V)$, we have
\[(U(x_n),V(x_n))\ge(u_{2x_n}(x_n),v_{2x_n}(x_n))\ge(u^*-\frac{\ep_0}{2},v^*-\frac{\ep_0}{2}).\]
 However, since $U(x_n)\le u^*-\ep_0$ for $n\ge1$, we derive a contradiction. So $(U(x),V(x))\to(u^*,v^*)$ as $x\to\yy$.

Let $(\tilde{U},\tilde{V})$ be another bounded positive solution of \eqref{2.4}. Similar to the arguments in the proof of the continuity of $(u_l,v_l)$ in the Step 1, we can show that $(\tilde{U},\tilde{V})$ is continuous in $x\ge0$. Moreover, by arguing as in the proof of the monotonicity of $(u_l,v_l)$ about $l$, we can deduce that $(\tilde{U},\tilde{V})\ge(u_l,v_l)$ for $x\in[0,l]$ and any large $l$, which, combined with the definition of $(U,V)$, yields that $(\tilde{U}(x),\tilde{V}(x))\ge(U(x),V(x))$ for $x\ge0$.

We now show that $(\tilde{U},\tilde{V})\le(u^*,v^*)$. By way of contradiction, we suppose that $\tilde{U}_{sup}:=\sup_{x\in[0,\yy)}\tilde{U}(x)>u^*$. Then there are two cases:
\begin{enumerate}[$(1)$]
  \item {\it Case 1:} There exists some $x_1\in[0,\yy)$ such that $\tilde{U}(x_1)=\tilde U_{sup}>u^*$.
  \item {\it Case 2:} $\tilde U(x)<\tilde{U}_{sup}$ for all $x\ge0$, and there exists a sequence converging to $\yy$ such that $\tilde{U}$ converges to $\tilde{U}_{sup}>u^*$ along this sequence.
\end{enumerate}

For Case 1, in view of \eqref{2.4}, we see $a\tilde U_{sup}\le H(\tilde{V}(x_1))$. Together with {\bf (H)}, we have $\tilde V_{sup}:=\sup_{x\in[0,\yy)}\tilde{V}(x)>v^*$. If there exists some $x_2\in[0,\yy)$ such that $\tilde{V}(x_2)=\tilde V_{sup}$, using \eqref{2.4} again leads to $b\tilde V_{sup}\le G(\tilde U(x_2))$. Hence we now derive
\[a\tilde U_{sup}\le H(\tilde{V}(x_1)), ~ b\tilde V_{sup}\le G(\tilde U(x_2)), ~ \tilde U_{sup}\ge\tilde{U}(x_2), ~ \tilde V_{sup}\ge\tilde{V}(x_1),\]
which, combined with {\bf (H)}, arrives at $G(H(\tilde{V}(x_1))/a)\ge b\tilde{V}(x_1)$. This implies that there exists another positive constant equilibrium which is different from $(u^*,v^*)$. So we obtain a contradiction. If $\tilde{V}(x)<\tilde V_{sup}$ for all $x\ge0$, then one can find a sequence $\{\tilde{x}_n\}\nearrow\yy$ such that $\tilde{V}(\tilde{x}_n)\to \tilde V_{sup}$ as $n\to\yy$. Since $\tilde{U}(\tilde{x}_n)$ is bounded, there exists a subsequence, still denoted by itself, such that $\tilde{U}(\tilde{x}_n)\to \tilde{U}^{\yy}$. Moreover, direct calculations show
\bess
\limsup_{n\to\yy}(\int_{0}^{\yy}J_2(\tilde{x}_n-y)\tilde{V}(y)\dy-\tilde{V}(\tilde{x}_n))\le\limsup_{n\to\yy}(\tilde V_{sup}-\tilde{V}(\tilde{x}_n))=0,
\eess
which implies that there exists a subsequence, still denoted by itself, such that
\[\lim_{n\to\yy}(b\tilde{V}(\tilde{x}_n)-G(\tilde{U}(\tilde{x}_n)))\le0.\]
Thus $b\tilde V_{sup}\le G(\tilde{U}^{\yy})$, and we have
\[a\tilde U_{sup}\le H(\tilde{V}(x_1)), ~ b\tilde V_{sup}\le G(\tilde{U}^{\yy}), ~ \tilde U_{sup}\ge\tilde{U}^{\yy}, ~ \tilde V_{sup}\ge\tilde{V}(x_1),\]
which similarly shows $G(H(\tilde{V}(x_1))/a)\ge b\tilde{V}(x_1)$. Then a analogous contradiction is obtained. To conclude, we now have derived that if $\tilde U_{sup}>u^*$, then $\tilde U_{sup}$ can not be achieved at some point in $[0,\yy)$, i.e., Case 1 can not happen.

 Arguing as above, we can show that if $\tilde U_{sup}>u^*$, then there is no such sequence along which $\tilde{U}$ converges to $\tilde U_{sup}$. Namely, Case 2 also will not occur. Therefore, $(\tilde{U},\tilde{V})\le(u^*,v^*)$ for all $x\ge0$. Combining with $(U,V)\to(u^*,v^*)$ as $x\to\yy$ and $(\tilde{U},\tilde{V})\ge(U,V)$, we immediately get $(\tilde U,\tilde V)\to(u^*,v^*)$ as $x\to\yy$.

With the aid of the above preparations, we are in the position to show the uniqueness. Clearly, it is sufficient to prove $(U,V)\ge(\tilde{U},\tilde{V})$ since we already showed $(U,V)\le(\tilde{U},\tilde{V})$. According to the above analysis, we can define
\[\rho^*=\inf\{\rho>1: \rho(U(x),V(x))\ge(\tilde{U}(x),\tilde{V}(x)) ~ ~ {\rm for ~ }x\ge0\}.\]
Obviously, $\rho^*\ge1$ and $\rho^*(U(x),V(x))\ge(\tilde{U}(x),\tilde{V}(x))$ for $x\ge0$. We now show $\rho^*=1$. Otherwise, it is easy to see that there exists some $x_2\in[0,\yy)$ such that $\rho^*U(x_2)=\tilde{U}(x_2)$ or $\rho^*V(x_2)=\tilde{V}(x_2)$.  Without loss of generality, we assume that $\rho^*U(x_2)=\tilde{U}(x_2)$. By virtue of \eqref{2.4}, we have $\rho^*H(V(x_2))\le H(\tilde{V}(x_2))$. However, by {\bf (H)}, we see
\bess
\rho^*H(V(x_2))-H(\tilde{V}(x_2))&=&\rho^*V(x_2)(\frac{H(V(x_2))}{V(x_2)}-\frac{H(\tilde{V}(x_2))}{\rho^*V(x_2)})\\
&\ge&\rho^*V(x_2)(\frac{H(V(x_2))}{V(x_2)}-\frac{H(\rho^*V(x_2))}{\rho^*V(x_2)})>0.
\eess
So we obtain a contradiction. Thus $\rho^*=1$, and further $(U(x),V(x))\ge(\tilde{U}(x),\tilde{V}(x))$ for $x\ge0$. So the uniqueness is proved.

{\bf Step 3:} {\it The monotonicity.} For any $x\ge0$ and $\delta>0$, denote $(U(x+\delta),V(x+\delta))$ by $(U_{\delta}(x),V_{\delta}(x))$. Then $(U_{\delta},V_{\delta})$ satisfies
\bess
\begin{cases}
 \dd d_1\int_{-\delta}^{\yy}J_1(x-y)U_{\delta}(y)\dy-d_1U_{\delta}-aU_{\delta}+H(V_{\delta})=0, ~ ~ x\in[0,\yy),\\
\dd d_2\int_{-\delta}^{\yy}J_2(x-y)V_{\delta}(y)\dy-d_2V_{\delta}-bV_{\delta}+G(U_{\delta})=0, ~ ~  x\in[0,\yy),
\end{cases}
\eess
which implies
\bess
\begin{cases}
 \dd d_1\int_{0}^{\yy}J_1(x-y)U_{\delta}(y)\dy-d_1U_{\delta}-aU_{\delta}+H(V_{\delta})\le0, ~ ~ x\in[0,\yy),\\
\dd d_2\int_{0}^{\yy}J_2(x-y)V_{\delta}(y)\dy-d_2V_{\delta}-bV_{\delta}+G(U_{\delta})\le0, ~ ~  x\in[0,\yy).
\end{cases}
\eess
Then we can argue as in the proof of the uniqueness to derive that $(U_{\delta}(x),V_{\delta}(x))\ge(U(x),V(x))$ for all $x\ge0$. Thus $(U,V)$ is increasing in $x\ge0$.

Now we show that $(U,V)$ is strictly increasing in $x\ge0$. We only prove the monotonicity of $U$ since one can similarly prove the case for $V$. Thanks to {\bf (J)}, there exists a $\delta>0$ such that $J_1(x)>0$ for $x\in[-\delta,\delta]$. Hence it is sufficient to show that $U$ is strictly increasing in $[k\delta,(k+1)\delta]$ for all nonnegative integer $k$. If there exist $x_1$ and $x_2$ with $0\le x_1<x_2\le \delta$ such that $U(x_1)=U(x_2)$, then by the equation of $U$, we have
\bes\label{2.8}
d_1\int_{0}^{\yy}J_1(x_1-y)U(y)\dy+H(V(x_1))=d_1\int_{0}^{\yy}J_1(x_2-y)U(y)\dy+H(V(x_2)).
\ees
A straightforward computation yields
\bess
&&\int_{0}^{\yy}J_1(x_2-y)U(y)\dy-\int_{0}^{\yy}J_1(x_1-y)U(y)\dy\\
&=&\int_{-x_2}^{\yy}J_1(y)U(y+x_2)\dy-\int_{-x_1}^{\yy}J_1(y)U(y+x_1)\dy\\
&=&\int_{-x_2}^{-x_1}J_1(y)U(y+x_2)\dy+\int_{-x_1}^{\yy}J_1(y)U(y+x_2)\dy-\int_{-x_1}^{\yy}J_1(y)U(y+x_1)\dy\\
&\ge&\int_{-x_2}^{-x_1}J_1(y)U(y+x_2)\dy>0,
\eess
which, combined with \eqref{2.8}, leads to $H(V(x_1))>H(V(x_2))$. However, since $V(x_1)\le V(x_2)$ and $H$ is strictly increasing in $x\ge0$, we must have $H(V(x_1))\le H(V(x_2))$. So we get a contradiction, and $U$ is strictly increasing in $[0,\delta]$.

Arguing inductively, we assume that $U$ is strictly increasing in $[k\delta,(k+1)\delta]$. If there exist $x_1$ and $x_2$ with $(k+1)\delta\le x_1<x_2\le(k+2)\delta$ such that $U(x_1)=U(x_2)$, then by the equation of $U$, \eqref{2.8} holds. Moreover, since $J_1(x)>0$ for $x\in[-\delta,\delta]$ and $U(x_2-\delta)-U(x_1-\delta)>0$, we have
\bess
&&\int_{0}^{\yy}J_1(x_2-y)U(y)\dy-\int_{0}^{\yy}J_1(x_1-y)U(y)\dy\\
&=&\int_{-x_2}^{\yy}J_1(y)U(y+x_2)\dy-\int_{-x_1}^{\yy}J_1(y)U(y+x_1)\dy\\
&=&\int_{-x_2}^{-x_1}J_1(y)U(y+x_2)\dy+\int_{-x_1}^{\yy}J_1(y)U(y+x_2)\dy-\int_{-x_1}^{\yy}J_1(y)U(y+x_1)\dy\\
&\ge&\int_{-\delta}^{0}J_1(y)(U(y+x_2)-U(y+x_1))\dy>0.
\eess
Then as above, we can derive a contradiction. Thus $U$ is strictly increasing in $[0,\yy)$. This step is complete.

Now to finish the proof, it remains to show $(U,V)<(u^*,v^*)$ for $x\ge0$. Recall that we already knew $(U,V)\le(u^*,v^*)$ for $x\ge0$. If $U$ achieves its supremum $u^*$ somewhere on $[0,\yy)$, then there must exist some $x_0$ such that $U(x_0)=u^*$ and $U(x)<u^*$ in the left or right neighborhood of $x_0$ since $U\not\equiv u^*$.
Direct calculations show that
\[\int_{0}^{\yy}J(x_0-y)U(y)\dy-U(x_0)<0.\]
Substituting such $x_0$ into the equation of $U$ leads to $H(V(x_0))>aU(x_0)=au^*$ which implies $V(x_0)>v^*$. This is a contradiction since $V(x)\le v^*$ for all $x\ge0$. Analogously, we can show $V<v^*$. Therefore, the proof is ended.
\end{proof}

\begin{remark}\label{r2.1}From the above proof, it follows that if $\lambda_p(l)>0$, then \eqref{2.5} has a positive solution $(u_l,v_l)$. Moreover, by arguing as in the proof of uniqueness, we can show that the positive solution $(u_l,v_l)$ of \eqref{2.5} is unique.
\end{remark}

Next we discuss the following two problems whose spatial domains are a bounded domain $[0,l]$ and a half space $[0,\yy)$, respectively.
\bes\left\{\begin{aligned}\label{2.9}
&u_t=d_1\int_0^lJ_1(x-y)u(y)\dy-d_1u-au+H(v), & & t>0, ~ ~ x\in[0,l]\\
&v_t=d_2\int_0^lJ_2(x-y)v(y)\dy-d_2v-bv+G(u), & & t>0, ~ ~ x\in[0,l],\\
&u(0,x)=\tilde{u}_0(x), ~ ~v(0,x)=\tilde{v}_0(x),
 \end{aligned}\right.
 \ees
where $(\tilde u_0,\tilde v_0)\in X^+\setminus\{(0,0)\}$.

\bes\left\{\begin{aligned}\label{2.10}
&u_t=d_1\int_0^{\yy}J_1(x-y)u(y)\dy-d_1u-au+H(v), & & t>0, ~ ~ x\in[0,\yy)\\
&v_t=d_2\int_0^{\yy}J_2(x-y)v(y)\dy-d_2v-bv+G(u), & & t>0, ~ ~ x\in[0,\yy),\\
&u(0,x)=\tilde{u}_0(x), ~ ~v(0,x)=\tilde{v}_0(x),
 \end{aligned}\right.
 \ees
where $(\tilde u_0,\tilde v_0)\in C([0,\yy))\cap L^{\yy}([0,\yy))$ and $(\tilde u_0,\tilde v_0)$ is negative and not identically equal to $(0,0)$.

\begin{proposition}\label{p2.3}The following statements are valid.
\begin{enumerate}[$(1)$]
  \item Problem \eqref{2.9} has a unique global solution. If $\lambda_p(l)>0$, then the unique solution of \eqref{2.9} will converge to the unique positive solution $(u_l,v_l)$ of \eqref{2.5} in $C([0,l])$ as $t\to\yy$; if $\lambda_p(l)\le0$, then the unique solution of \eqref{2.9} will converge to $(0,0)$ in $C([0,l])$ as $t\to\yy$;
  \item Problem \eqref{2.10} has a unique global solution, which converges to the unique bounded positive solution $(U,V)$ of \eqref{2.4} in $C_{{\rm loc}}([0,\yy))$ as $t\to\yy$.
\end{enumerate}
\end{proposition}

\begin{proof}
By the fixed point theory and maximum principle, it is easy to show the existence and uniqueness of global solutions for \eqref{2.9} and \eqref{2.10}. For the longtime behaviors, one can use the similar arguments as in the proof of \cite[Proposition 3.4]{DN8} to derive the desired results. The details are omitted here.
\end{proof}

\section{Dynamics of \eqref{1.1}}
In this section, we are going to investigate the dynamics of \eqref{1.1} by using the results in previous section. Firstly, by following the similar lines as in \cite{NV,DN8}, we can prove the well-posedness of \eqref{1.1}, i.e., Theorem \ref{t1.1}. Since the modifications are obvious, we omit the details. Owing to $h'(t)>0$ for all $t\ge0$, $h(t)$ will converge to a finite value or infinity as $t\to\yy$. Set $h_{\yy}=\lim_{t\to\yy}h(t)$. We call the case $h_{\yy}=\yy$ {\it spreading}, and the case $h_{\yy}<\yy$ {\it vanishing}.

We first discuss the longtime behaviors of solution component $(u,v)$ of \eqref{1.1}, which will be divided into two cases: {\it spreading} and {\it vanishing }. Let $\lambda_p(l)$ be the principal eigenvalue of \eqref{2.1} on the domain $[0,l]$.

\begin{lemma}\label{l3.1}If vanishing happens, i.e., $h_{\yy}<\yy$, then $\lambda_p(h_{\yy})\le0$ and
\[\lim_{t\to\yy}\|u(t,x)+v(t,x)\|_{C([0,h(t)])}=0.\]
\end{lemma}

\begin{proof}We first prove that if $h_{\yy}<\yy$, then $\lambda_p(h_{\yy})\le0$. Assume on the contrary that $\lambda_p(h_{\yy})>0$. By the continuity of $\lambda_p(l)$ in $l$, there exist small $\ep>0$ and $\delta>0$ such that $\lambda_p(h_{\yy}-\ep)>0$ and $\min\{J_1(x),J_2(x)\}\ge\delta$ for $|x|\le 2\ep$ due to {\bf (J)}.  Moreover, there is $T>0$ such that $h(t)>h_{\yy}-\ep$ for $t\ge T$. Hence the solution component $(u,v)$ of \eqref{1.1} satisfies
\bess\left\{\begin{aligned}
&u_t\ge d_1\int_0^{h_{\yy}-\ep}J_1(x-y)u(y)\dy-d_1u-au+H(v), & & t>T, ~ ~ x\in[0,h_{\yy}-\ep]\\
&v_t\ge d_2\int_0^{h_{\yy}-\ep}J_2(x-y)v(y)\dy-d_2v-bv+G(u), & & t>T, ~ ~ x\in[0,h_{\yy}-\ep],\\
&u(T,x)>0, ~ ~v(T,x)>0, & & x\in[0,h_{\yy}-\ep].
 \end{aligned}\right.
 \eess
Let $(\underline{u},\underline{v})$ be the unique solution of \eqref{2.9} with $l=h_{\yy}-\ep$, $\tilde{u}_0(x)=u(T,x)$ and $\tilde{v}_0(x)=v(T,x)$. Note that $\lambda_p(h_{\yy}-\ep)>0$. Making use of Proposition \ref{p2.3} we have $(\underline{u},\underline{v})\to(u_l,v_l)$ in $X$ as $t\to\yy$, where $(u_l,v_l)$ is the unique positive solution of \eqref{2.5} with $l=h_{\yy}-\ep$.

By a comparison argument, we see $u(t+T,x)\ge \underline{u}(t,x)$ and $v(t+T,x)\ge \underline{v}(t,x)$ for $x\in[0,h_{\yy}-\ep]$. Therefore, $\liminf_{t\to\yy}(u(t,x),v(t,x))\ge(u_l,v_l)$ uniformly in $[0,h_{\yy}-\ep]$. Moreover, there exist small $\sigma>0$ and large $T_1\gg T$ such that $u(t,x)\ge\sigma$ and $v(t,x)\ge\sigma$ for $t\ge T_1$ and $[0,h_{\yy}-\ep]$. In view of the equation of $h(t)$, we have, for $t>T_1$,
 \bess
 h'(t)\ge\sigma\int_{h_{\yy}-\frac{3\ep}2}^{h_{\yy}-\ep}\int_{h_{\yy}}^{h_{\yy}
 +\frac{\ep}2}\big[\mu_1J_1(x-y)+\mu_2J_2(x-y)\big]\dy\dx\ge(\mu_1+\mu_2)\delta\sigma,
 \eess
which clearly contradicts $h_{\yy}<\yy$. Thus $\lambda_p(h_{\yy})\le0$.

 Let $(\bar{u},\bar{v})$ be the solution of \eqref{2.9} with $l=h_{\yy}$, $\tilde{u}_0(x)=\|u_0\|_{C([0,h_0])}$ and $\tilde{v}_0(x)=\|v_0\|_{C([0,h_0])}$. Clearly, $\bar{u}(t,x)\ge u(t,x)$ and $\bar{v}\ge v(t,x)$ for $t\ge0$ and $x\in[0,h(t)]$. Note that $\lambda_p(h_{\yy})\le0$. Thus $(\bar{u},\bar{v})\to(0,0)$ uniformly in $[0,l]$ as $t\to\yy$, which implies that $(u,v)\to(0,0)$ uniformly in $[0,h(t)]$ as $t\to\yy$. The proof is ended.
\end{proof}

Then we turn to the case {\it spreading}.

\begin{lemma}\label{l3.2}If spreading happens, namely, $h_{\yy}=\yy$, then
\[\lim_{t\to\yy}(u(t,x),v(t,x))=(U(x),V(x)) ~ ~ {\rm locally ~ uniformly ~ in ~ }[0,\yy),\]
where $(U,V)$ is the unique bounded positive solution of \eqref{2.4}.
\end{lemma}
\begin{proof}Since $h_{\yy}=\yy$, we have $\mathcal{R}_0>1$ (see Lemma \ref{l3.3} below).
Let $(\bar{u},\bar{v})$ be the unique global solution of \eqref{2.10} with $(\tilde{u}_0,\tilde{v}_0)=(\|u_0\|_{C([0,h_0])},\|v_0\|_{C([0,h_0])})$. Then using a comparison consideration yields that $\bar{u}(t,x)\ge u(t,x)$ and $\bar{v}\ge v(t,x)$ for $t\ge0$ and $x\in[0,h(t)]$. Moreover, in view of Proposition \ref{p2.3}, we have $(\bar{u},\bar{v})\to(U,V)$ locally uniformly in $[0,\yy)$ as $t\to\yy$. So we obtain
\bes\label{l3.1}
\limsup_{t\to\yy}(u(t,x),v(t,x))\le (U(x),V(x)) {\rm ~ ~ locally ~ uniformly ~ in ~ }[0,\yy).
\ees

It now remains to the lower limit of $(u,v)$. Notice that the unique positive solution $(u_l,v_l)$ of \eqref{2.5} converges to $(U,V)$ locally uniformly in $[0,\yy)$ as $l\to\yy$. For any $l_0>0$ and small $\ep>0$, there exists a large $L>0$ such that for all $l\ge L$,
\[(u_l(x),v_l(x))\ge(U(x)-\frac{\ep}{2},V(x)-\frac{\ep}{2}) {\rm ~ ~ for ~ }x\in[0,l_0].\]
Moreover, due to $h_{\yy}=\yy$, there exists $T>0$ such that $h(t)>L$ for $t\ge T$. Hence the solution component $(u,v)$ of \eqref{1.1} satisfies
\bess\left\{\begin{aligned}
&u_t\ge d_1\int_0^{L}J_1(x-y)u(y)\dy-d_1u-au+H(v), & & t>T, ~ ~ x\in[0,L]\\
&v_t\ge d_2\int_0^{L}J_2(x-y)v(y)\dy-d_2v-bv+G(u), & & t>T, ~ ~ x\in[0,L],\\
&u(T,x)>0, ~ ~v(T,x)>0, & & x\in[0,L].
 \end{aligned}\right.
 \eess
Let $(\underline{u},\underline{v})$ be the unique solution of \eqref{2.9} with $l=L$, $\tilde{u}_0(x)=u(T,x)$ and $\tilde{v}_0(x)=v(T,x)$. Note that $L$ is large enough such that $\lambda_p(L)>0$. By Proposition \ref{p2.3}, we have that as $t\to\yy$,
\[(\ud u(t,x),\ud v(t,x))\to(u_L(x),v_L(x)) {\rm ~ ~ uniformly ~ in ~ }[0,L].\]
Thus there exists $T_1\gg T$ such that for $t\ge T_1$,
\[(\ud u(t,x),\ud v(t,x))\ge(u_L(x)-\frac{\ep}{2},v_L(x)-\frac{\ep}{2}) {\rm ~ ~ uniformly ~ in ~ }[0,L].\]
Moreover, from a comparison argument, it follows that $u(t+T,x)\ge \underline{u}(t,x)$ and $v(t+T,x)\ge \underline{v}(t,x)$ for $t\ge0$ and $x\in[0,L]$. Hence for $t\ge T_1$, we have
\[(u(t,x),v(t,x))\ge(U(x)-\ep,V(x)-\ep) {\rm ~ ~ in ~ }[0,l_0],\]
which, together with the arbitrariness of $\ep$, yields that
\[\liminf_{t\to\yy}(u(t,x),v(t,x))\ge(U(x),V(x)) {\rm ~ ~ locally ~ uniformly ~ in ~ }[0,\yy).\]
Therefore, the proof is finished.
\end{proof}

Clearly, Theorem \ref{t1.2} follows from Lemmas \ref{l3.1}-\ref{l3.2}.

Now we study the criteria governing spreading and vanishing.

\begin{lemma}\label{l3.3} If $\mathcal{R}_0\le1$, then vanishing happens. Particularly,
 \bess
 h_{\yy}\le h_0+\frac 1{\min\kk\{d_1/\mu_1,\, H'(0)d_2/(b\mu_2)\rr\}}\int_0^{h_0}\left(u_0(x)
+\frac{H'(0)}{b}v_0(x)\right)\dx.\eess
\end{lemma}

\begin{proof} It is easy to see that
 \bess
&& \int_0^{h(t)}\!\!\int_0^{h(t)}\!J_1(x-y)u(t,y)\dy\dx-\int_0^{h(t)}\!u(t,x)\dx\\
&=& \int_0^{h(t)}\!\!\int_0^{h(t)}\!J_1(x-y)u(t,y)\dy\dx-\int_0^{h(t)}\int_{-\yy}^{\yy}\!J_1(x-y)u(t,x)\dy\dx\\
&=&-\int_0^{h(t)}\int_{-\yy}^{0}\!J_1(x-y)u(t,x)\dy\dx-\int_0^{h(t)}\!\!\int_{h(t)}^\yy\!J_1(x-y)u(t,x)\dx\dy\\
&\le&-\int_0^{h(t)}\!\!\int_{h(t)}^\yy\!J_1(x-y)u(t,x)\dx\dy.
\eess
Similarly,
\bess
\int_0^{h(t)}\!\!\int_0^{h(t)}\!J_2(x-y)v(t,y)\dy\dx-\int_0^{h(t)}\!v(t,x)\dx
&\le&-\int_0^{h(t)}\!\!\int_{h(t)}^\yy\!J_2(x-y)v(t,x)\dx\dy.
 \eess
 Note that $\mathcal{R}_0<1$. Simple calculations show that for $t>0$ and $x\in[0,h(t)]$,
 \[H(v(t,x))-au(t,x)-H'(0)v(t,x)+\frac{H'(0)}{b}G(u(t,x))\le0.\]

Hence, after a series of simple computations, we have
\bess
 &&\frac{\rm d}{{\rm d}t}\!\int_0^{h(t)}\!\!\left(u(t,x)+\frac{H'(0)}{b}v(t,x)\right)\dx\\
&\le&-\int_0^{h(t)}\!\!\!\int_{h(t)}^{\yy}\!\!\kk(d_1J_1(x-y)u(t,x)+\frac{H'(0)d_2}{b}
J_2(x-y)v(t,x)\!\rr)\!\dy\dx\\[1mm]
&&+\int_0^{h(t)}\kk(H(v(t,x))-au(t,x)-H'(0)v(t,x)+\frac{H'(0)}{b}G(u(t,x))\rr)\dx\\[1mm]
&<&-\min\kk\{\frac{d_1}{\mu_1},\, \frac{H'(0)d_2}{b\mu_2}\rr\}h'(t).
\eess
Integrating the above inequality from $0$ to $t$ completes the proof.
\end{proof}
Next we consider the case $\mathcal{R}_0>1$. All arguments used below tightly depend on the fact that if vanishing happens, then $\lambda_p(h_{\yy})\le0$ as in Lemma \ref{l3.1}. Here we recall
\bess
&&\gamma_A=\frac{-a-b+\sqrt{(a-b)^2
+4H'(0)G'(0)}}2,\\[1mm]
&&\gamma_B=\frac{-a-d_1-b-d_2
+\sqrt{(a+d_1-b-d_2)^2+4H'(0)G'(0)}}2.
\eess
It is clear that $\gamma_B\ge0$ if and only if
\[\mathcal{R}_*:=\frac{H'(0)G'(0)}{(a+d_1)(b+d_2)}\ge1.\]

\begin{lemma}\label{l3.4}If $\mathcal{R}_*\ge1$, then spreading occurs.
\end{lemma}

\begin{proof} As we see above, the condition $\mathcal{R}_*\ge1$ implies $\gamma_B\ge 0$.
Owing to Proposition \ref{p2.1}, we have $\lambda_p(l)>\gamma_B\ge 0$ for all $l>0$. It then follows from Lemma \ref{l3.1} that spreading happens. The proof is ended.
\end{proof}

In what follows, we focus on the case $\mathcal{R}_*<1<\mathcal{R}_0$. Making use of Proposition \ref{p2.1}, we have $\lim_{l\to\yy}\lambda_p(l)=\gamma_A>0$, and $\lim_{l\to 0}\lambda_p(l)=\gamma_B <0$. By the monotonicity of $\lambda_p(l)$, there exists a unique $\ell^*>0$ such that $\lambda_p(\ell^*)=0$ and $\lambda_p(l)(l-\ell^*)>0$ for $l\neq\ell^*$. Then as a consequence of Lemma \ref{l3.1}, we have the following result.

\begin{lemma}\label{l3.5}Let $\ell^*$ be defined as above. If $h_0\ge\ell^*$, then spreading happens.
\end{lemma}

The next result shows that if $h_0<\ell^*$ and $\mu_1+\mu_2$ small enough, then vanishing occurs. Before stating this result, we need a comparison principle for \eqref{1.1}. Since its proof is similar to those of \cite[Lemma 3.2]{DN8}, we omit the details here.

\begin{proposition}[Comparison principle]\label{p3.1}\, Let $J_i$ satisfy the condition {\bf(J)} for $i=1,2$. If $\bar{h}\in C^1([0,T])$, $\bar{u},\,\bar{v},\, \bar{u}_{t}, \,\bar{v}_t\in C([0,T]\times[0,\bar{h}(t)])$ and satisfy
\bess
\left\{\begin{aligned}
&\bar u_{t}\ge d_1\dd\int_{0}^{\bar h(t)}J_1(x-y)\bar u(t,y)\dy-d_1\bar u-a\bar{u}+H(\bar{v}), && 0<t\le T,~0\le x<\bar{h}(t),\\
&\bar v_t\ge d_2\dd\int_{0}^{\bar h(t)}J_2(x-y)\bar v(t,y)\dy-d_2\bar v-b\bar{v}+G(\bar{u}), && 0<t\le T,~0\le x<\bar{h}(t),\\
&\bar u(t,\bar{h}(t))\ge0, ~ \bar{v}(t,h(t))\ge0 &&0<t\le T,\\
&\bar h'(t)\ge\int_0^{\bar h(t)}\!\int_{\bar h(t)}^{\yy}\big[\mu_1 J_1(x-y)\bar u(t,x)
+\mu_2 J_2(x-y)\bar v(t,x)\big]\dy\dx, &&0<t\le T,\\
&\bar h(0)\ge h_0,\;\;\bar u(0,x)\ge u_{0}(x),\;\;\bar{v}(0,x)\ge v_0(x),&& 0\le x\le h_0,
\end{aligned}\right.
 \eess
 then the unique solution $(u,v,h)$ of \eqref{1.1} satisfies
 \[u(t,x)\le\bar{u}(t,x), ~ v(t,x)\le \bar{v}(t,x), ~ h(t)\le\bar{h}(t) ~ ~ {\rm for} ~ 0\le t\le T, ~ 0\le x\le h(t).\]
\end{proposition}

According to Proposition \ref{p3.1}, we can see that the unique solution $(u,v,h)$ of \eqref{1.1} is increasing in $\mu_1$ and $\mu_2$, respectively.

\begin{lemma}\label{l3.6}If $h_0<\ell^*$, then there exists a $\ud\mu>0$ such that vanishing happens if $\mu_1+\mu_2\le\ud\mu$.
\end{lemma}

\begin{proof} Due to $h_0<\ell^*$, we have $\lambda_p(h_0)<0$. By the continuity of $\lambda_p(l)$, there exists a small $\ep>0$ such that $\lambda_p(h_0(1+\ep))<0$. For convenience, denote $h_1=h_0(1+\ep)$. Let $\phi=(\phi_1,\phi_2)$ be the positive eigenfunction of $\lambda_p(h_0(1+\ep))$ with $\|\phi\|_{X}=1$. Define \[\bar{h}(t)=h_0\big[1+\ep(1-{\rm e}^{-\delta t})\big], ~ \bar{u}(t,x)=M{\rm e}^{-\delta t}\phi_1, ~ \bar{v}=M{\rm e}^{-\delta t}\phi_2\]
with $0<\delta\le -\lambda_p(h_1)$ and $M$ large enough such that $M\phi_1(x)\ge u_0(x)$ and $M\phi_2(x)\ge v_0(x)$ for $x\in[0,h_1]$.

Direct calculations yield that, for $t>0$ and $x\in[0,\bar{h}(t)]$,
 \bess
 &&\bar{u}_t-d_1\int_0^{\bar{h}(t)}\!\!J_1(x-y)\bar{u}(t,y)\dy
+d_1\bar{u}+a\bar{u}-H(\bar{v})\\
 &\ge& M{\rm e}^{-\delta t}\left(-\delta\phi_1-d_1\int_0^{h_1}J_1(x-y)\phi_1(y)\dy+d_1\phi_1+a\phi_1-\frac{H(\bar{v})}{M{\rm e}^{-\delta t}}\right)\\
 &=&M{\rm e}^{-\delta t}\left(-\delta\phi_1-\lambda_p(h_1)\phi_1+H'(0)\phi_2-\frac{H(M{\rm e}^{-\delta t}\phi_2)}{M{\rm e}^{-\delta t}}\right)\\
 &\ge& M{\rm e}^{-\delta t}\left(-\delta-\lambda_p(h_1)\right)\phi_1\ge 0.
\eess
Similarly,
\[\bar{v}_t-d_2\int_0^{\bar{h}}J_2(x-y)\bar{v}(t,y)\dy
+d_2)\bar{v}+b\bar{v}-G(\bar{u})\ge0.\]
Moreover, when $\mu_1+\mu_2\le\frac{\ep\delta h_0}{Mh_1}$, we have
 \bess
&&\int_0^{\bar{h}(t)}\!\!\int_{\bar{h}(t)}^{\yy}\big[\mu_1J_1(x-y)\bar{u}(t,x)
+\mu_2J_2(x-y)\bar{v}(t,x)\big]\dy\dx\\
&=& M{\rm e}^{-\delta t}\sum_{i=1}^2\mu_i\int_0^{\bar{h}(t)}\!\!\int_{\bar{h}(t)}^{\yy}J_i(x-y)\phi_i(x)\dy\dx\\
&\le&(\mu_1+\mu_2)Mh_1{\rm e}^{-\delta t}\le\ep\delta h_0{\rm e}^{-\delta t}=\bar{h}'(t).
  \eess

By the comparison principle (Proposition \ref{p3.1}), we have $h(t)\le\bar{h}(t)$ for $t\ge0$, which implies
\[\lim_{t\to\yy}h(t)\le\lim_{t\to\yy}\bar h(t)=h_1<\yy.\]
Thus the proof is finished.
\end{proof}

The next result shows that if $h_0<\ell^*$, spreading can happen provided that one of the expanding rates $\mu_1$ and $\mu_2$ is large enough.

\begin{lemma}\label{l3.7}If $h_0<\ell^*$, then there exists a $\bar{\mu}_1>0$\, $(\bar{\mu}_2>0)$ which is independent of $\mu_2$ $(\mu_1)$ such that spreading happens when $\mu_1\ge\bar{\mu}_1$ $(\mu_2\ge\bar{\mu}_2)$.
\end{lemma}

\begin{proof}We only prove the assertion about $\mu_1$ since the similar method can be adopt for the conclusion of $\mu_2$.
Let $(\underline{u},\underline{v},\underline{h})$ be the unique solution of \eqref{1.1} with $\mu_2=0$. Clearly, $(\underline{u},\underline{v},\underline{h})$ is a lower solution of \eqref{1.1} and Lemmas \ref{l3.1}-\ref{l3.6} hold for $(\underline{u},\underline{v},\underline{h})$. Then we can argue as in the proof of \cite[Theorem 1.3]{DN8} to deduce that there exists a $\ol{\mu}_1>0$ such that if $\mu_1\ge\ol{\mu}_1$, spreading happens for $(\underline{u},\underline{v},\underline{h})$ and also for the unique solution $(u,v,h)$ of \eqref{1.1}. The proof is finished.
\end{proof}

By Lemmas \ref{l3.6} and \ref{l3.7}, we have that vanishing occurs if $\mu_1+\mu_2\le \ud{\mu}$, while spreading happens if $\mu_1+\mu_2\ge\ol{\mu}_1+\ol{\mu}_2=:\ol\mu$. One naturally wonders whether these is a unique critical value of $\mu_1+\mu_2$ such that spreading happens if and only if $\mu_1+\mu_2$ is beyond this critical value. Indeed, such value does not exist since the unique solution $(u,v,h)$ of \eqref{1.1} is not monotone about $\mu_1+\mu_2$. However, for some special $(\mu_1, \mu_2)$ we can obtain a unique critical value as we wanted.

\begin{lemma}\label{l3.8} Assume $h_0<\ell^*$. If $\mu_2=f(\mu_1)$ where $f\in C([0,\yy))$, is strictly increasing, $f(0)=0$ and $\lim\limits_{s\to\yy}f(s)=\yy$. Then there is a unique $\mu^*_1>0$ such that spreading occurs if and only if $\mu_1>\mu^*_1$.
\end{lemma}

\begin{proof} Firstly, it is easy to see from a comparison argument that the unique solution $(u,v,h)$ is strictly increasing in $\mu_1$. We have known that vanishing happens when $\mu_1+f(\mu_1)\le\ud{\mu}$ (Lemma \ref{l3.6}), and spreading happens when $\mu_1+f(\mu_1)\ge\ol{\mu}$ (Lemma \ref{l3.7}). Due to the properties of $f$, there exist unique $\ud\mu_1$ and $\ol\mu_1>0$,  such that $\ud\mu_1+f(\ud\mu_1)=\ud{\mu}$ and $\ol\mu_1+f(\ol\mu_1)=\ol{\mu}$. Clearly, $\mu_1+f(\mu_1)\le\ud{\mu}$ is equivalent to $\mu_1\le\ud{\mu}_1$, and $\mu_1+f(\mu_1)\ge\ol{\mu}$ is equivalent to $\mu_1\ge\ol{\mu}_1$. So, vanishing happens if $\mu_1\le\ud{\mu}_1$, while spreading occurs if $\mu_1\ge\ol{\mu}_1$. Then we can use the monotonicity of $(u,v,h)$ on $\mu_1$ and argue as in the proof of \cite[Theorem 3.14]{CDLL} to finish the proof. The details are omitted here.
\end{proof}

Obviously, Theorem \ref{t1.3} can be obtained by Lemmas \ref{l3.3}-\ref{l3.8}.


\begin{thebibliography}{99}
\bibliographystyle{siam}
\setlength{\baselineskip}{15pt}

\vspace{-1.5mm}\bibitem{HY} C.-H. Hsu and T.-S. Yang, {\it Existence, uniqueness, monotonicity and asymptotic behaviour of travelling waves for epidemic models}, Nonlinearity, \textbf{26} (2013), 121-139.

\vspace{-1.5mm}\bibitem{DL}Y.H. Du and Z.G. Lin, {\it Spreading-vanishing dichotomy in the diffusive logistic model with a free boundary}, SIAM J. Math. Anal., \textbf{42}  (2010), 377-405.

\vspace{-1.5mm}\bibitem{CDLL} J.-F. Cao, Y.H. Du, F. Li and W.-T. Li, {\it The dynamics of a Fisher-KPP nonlocal diffusion model with free boundaries},  J. Funct. Anal., \textbf{277} (2019), 2772-2814.

\vspace{-1.5mm}\bibitem{LLW}L. Li, W.-T. Li and M.X. Wang, {\it Dynamics for nonlocal diffusion problems with a free boundary}, J. Differential Equations, \textbf{330}  (2022), 110-149.

\vspace{-1.5mm}\bibitem{LL}X.P. Li, L. Li, Y. Xu and D.D. Zhu, {\it Dynamics for a diffusive epidemic model with a free boundary: spreading-vanishing dichotomy}, submitted, (2024), arXiv: 2405.01070.

\vspace{-1.5mm}\bibitem{LLW1}X.P. Li, L. Li, and M. X. Wang, {\it Dynamics for a diffusive epidemic model with a free boundary: spreading speed}, submitted, (2024), 	arXiv:2407.08928.

\vspace{-1.5mm}\bibitem{LLW2}X.P. Li, L. Li, and M. X. Wang, {\it Dynamics for a diffusive epidemic model with a free boundary: sharp asymptotic profile}, submitted, (2024), 	arXiv:2407.11702.

\vspace{-1.5mm}\bibitem{AMRT}F. Andreu, J.M. Maz{\'o}n, J.D. Rossi and J. Toledo, {\it Nonlocal Diffusion Problems}, Math. Surveys Monogr., \textbf{165}, American Mathematical Society, Providence, RI, 2010.

\vspace{-1.5mm}\bibitem{KLS}C.-Y. Kao, Y. Lou and W.X. Shen, {\it Random dispersal vs. non-local dispersal}, Discrete Contin. Dyn. Syst., \textbf{26} (2010), 551-596.

\vspace{-1.5mm}\bibitem{BCV}H. Berestycki, J. Coville and H.-H. Vo, {\it Persistence criteria for populations with non-local dispersion}, J. Math. Biol., \textbf{72}  (2016), 1693-1745.

\vspace{-1.5mm}\bibitem{ZhangW24}Q.Y. Zhang and M.X Wang, {\it The nonlocal dispersal equation with seasonal succession}, Proc. Amer. Math. Soc., \textbf{152} (2024), 1083-1097.

\vspace{-1.5mm}\bibitem{Ya}H. Yagisita, {\it Existence and nonexistence of traveling waves for a nonlocal monostable equation}, Publ. Res. Inst. Math. Sci., \textbf{45}  (2009), 925-953.

\vspace{-1.5mm}\bibitem{Gar}J. Garnier, {\it Accelerating solutions in integro-differential equations}, SIAM J. Math. Anal., \textbf{43} (2011), 1955-1974.

\vspace{-1.5mm}\bibitem{AC}M. Alfrao and J. Coville, {\it Propagation phenomena in monostable integro-differential equations: acceleration or not?}, J. Differential Equations, \textbf{263} (2017), 5727-5758.

\vspace{-1.5mm}\bibitem{XLR}W.-B. Xu, W.-T. Li and S.G. Ruan, {\it Spatial propagation in nonlocal dispersal Fisher-KPP equations}, J. Funct. Anal., \textbf{280} (2020), 108957.

\vspace{-1.5mm}\bibitem{CQW}C. Cort{\'a}zar, F. Quir{\'o}s and N. Wolanski, {\it A nonlocal diffusion problem with a sharp free boundary}, Interfaces Free Bound., \textbf{21} (2019), 441-462.

\vspace{-1.5mm}\bibitem{DLZ}Y.H. Du, F. Li and M.L. Zhou, {\it Semi-wave and spreading speed of the nonlocal Fisher-KPP equation with free boundaries}, J. Math. Pures Appl., \textbf{154} (2021), 30-66.

\vspace{-1.5mm}\bibitem{DMZ}Y.H. Du, H. Matsuzawa and M.L. Zhou, {\it Sharp estimate of the spreading speed determined by nonlinear free boundary problems}, SIAM J. Math. Anal., \textbf{46} (2014), 375-396.

\vspace{-1.5mm}\bibitem{DLou}Y.H. Du and B.D. Lou, {\it Spreading and vanishing in nonlinear diffusion problems with free boundaries}, J. Eur. Math. Soc., \textbf{17} (2015), 2673-2724.

\vspace{-1.5mm}\bibitem{WND} Z.G. Wang, H. Nie and Y.H. Du, {\it Spreading speed for a West Nile virus model with free boundary}, J. Math. Biol., \textbf{79} (2019), 433-466.

\vspace{-1.5mm}\bibitem{DWu} Y.H. Du and C.-H. Wu, {\it Classification of the spreading behaviors of a two-species diffusion-competition system with free boundaries}, Cal. Var. PDE., \textbf{61} (2022), Article No. 54.

\vspace{-1.5mm}\bibitem{WQW} Z.G. Wang, Q. Qin and J.H. Wu, {\it Spreading speed and profile for the Lotka–Volterra competition model with two free boundaries}, J. Dyn. Diff. Equat., (2022), https://doi.org/10.1007/s10884-022-10222-6.

\vspace{-1.5mm}\bibitem{DDL}W.W. Ding, Y.H. Du and X. Liang, {\it Spreading in space-time periodic media governed by a monostable equation with free boundaries, Part 2: Spreading speed}, AIHP Analyse non Lineaire, \textbf{36} (2019), 1539-1573.

\vspace{-1.5mm}\bibitem{ZZh}Z.-C. Zhang and X.L. Zhang, {\it Asymptotic behavior of solutions for a free boundary problem with a nonlinear gradient absorption}, Cal. Var. PDE.,  \textbf{58} (2019), Article No. 32.

\vspace{-1.5mm}\bibitem{Du22}Y.H. Du, {\it Propagation and reaction-diffusion models with free boundaries}, Bull. Math. Sci., {\bf 12} (2022), Article No. 2230001.

\vspace{-1.5mm}\bibitem{DN2}Y.H. Du and W.J. Ni, {\it Spreading speed for some cooperative systems with nonlocal diffusion and free boundaries, part 1: Semi-wave and a threshold condition}, J. Differential Equations, \textbf{308} (2022), 369-420.

\vspace{-1.5mm}\bibitem{DN3}Y.H. Du and W.J. Ni, {\it The high dimensional Fisher-KPP nonlocal diffusion equation with free boundary and radial symmetry, Part 2}, (2022), submitted, arXiv: 2102.05286v1.

\vspace{-1.5mm}\bibitem{DN4}Y.H. Du and W.J. Ni, {\it Rate of propagation for the Fisher-KPP equation with nonlocal diffusion and free boundaries}, J. Eur. Math. Soc., (2023), Doi: 10.4171/JEMS/1392.

\vspace{-1.5mm}\bibitem{DN5}Y.H. Du and W.J. Ni, {\it Exact rate of accelerated propagation in the Fisher-KPP equation with nonlocal diffusion and free boundaries}, Math. Ann., \textbf{389} (2024), 2931-2958.

\vspace{-1.5mm}\bibitem{ZLZ}W.Y. Zhang, Z.H. Liu and L. Zhou, {\it Dynamics of a nonlocal diffusive logistic model with free boundaries in time periodic environment}, Discrete Contin. Dyn. Syst. B., \textbf{26} (2021), 3767-3784.

\vspace{-1.5mm}\bibitem{PLL} L.Q. Pu, Z.G. Lin and Y. Lou, {\it A West Nile virus nonlocal model with free boundaries and seasonal succession}, J. Math. Biol., \textbf{86} (2023), Article No. 25.

\vspace{-1.5mm}\bibitem{ZhangW23}Q.Y. Zhang and M.X. Wang, {\it A nonlocal diffusion competition model with seasonal succession and free boundaries}, Commun. Nonlinear Sci. Numer. Simul., \textbf{122} (2023), 107263.

\vspace{-1.5mm}\bibitem{DWZ22}Y.H. Du, M.X. Wang and M. Zhao, {\it Two species nonlocal diffusion systems with free boundaries}, Discrete Contin. Dyn. Syst., \textbf{42} (2022), 1127-1162.

\vspace{-1.5mm}\bibitem{LLW24} L. Li, X.P. Li and M.X. Wang, {\it The monostable cooperative system with nonlocal diffusion and free boundaries}, Proc. Royal Soc. Edinburgh A, \textbf{154} (2024), 629-659.

\vspace{-1.5mm}\bibitem{LW24} L. Li and M.X. Wang, {\it Free boundary problems of a mutualist model with nonlocal diffusion}, J. Dyn. Diff. Equat., \textbf{36} (2024), 375-403.

\vspace{-1.5mm}\bibitem{ZZLD}M. Zhao, Y. Zhang, W.-T. Li and Y.H. Du, {\it The dynamics of a degenerate epidemic model with nonlocal diffusion and free boundaries}, J. Differential Equations, \textbf{269} (2020), 3347-3386.

\vspace{-1.5mm}\bibitem{ZLD}M. Zhao, W.-T. Li and Y.H. Du, {\it The effect of nonlocal reaction in an epidemic model with nonlocal diffusion and free boundaries}, Comm. Pure Appl. Anal., \textbf{19} (2020), 4599-4620.

\vspace{-1.5mm}\bibitem{DLNZ}Y.H. Du, W.-T. Li, W.J. Ni and M. Zhao, {\it Finite or infinite spreading speed of an epidemic model with free boundary and double nonlocal effects}, J. Dyn. Diff. Equat., (2022), https://doi.org/10.1007/s10884-022-10170-1.

\vspace{-1.5mm}\bibitem{WD1}R. Wang and Y.H. Du, {\it Long-time dynamics of a nonlocal epidemic model with free boundaries: spreading-vanishing dichotomy}, J. Differential Equations, \textbf{327} (2022), 322-381.

\vspace{-1.5mm}\bibitem{WD2}R. Wang and Y.H. Du, {\it Long-time dynamics of a nonlocal epidemic model with free boundaries: spreading speed}, Discrete. Contin. Dyn. Syst.,  \textbf{43} (2023), 121-161.

\vspace{-1.5mm}\bibitem{DNW}Y.H. Du, W.J. Ni and R. Wang, {\it Rate of accelerated expansion of the epidemic region in a nonlocal epidemic model with free boundaries}, Nonlinearity, \textbf{36} (2023), 5621-5660.

\vspace{-1.5mm}\bibitem{NV} T.-H. Nguyen and H.-H. Vo, {\it Dynamics for a two-phase free boundaries system in an epidemiological model with couple nonlocal dispersals}, J. Differential Equations, \textbf{335} (2022), 398-463.

\vspace{-1.5mm}\bibitem{SWZ}Y.-H. Su, X.F. Wang and T. Zhang, {\it Principal spectral theory and variational characterizations for cooperative systems with nonlocal and coupled diffusion}, J. Differential Equations, \textbf{369} (2023), 94-114.

\vspace{-1.5mm}\bibitem{DN8}Y.H. Du and W.J. Ni, {\it Analysis of a West Nile virus model with nonlocal diffusion and free boundaries}, Nonlinearity, \textbf{33} (2020), 4407-4448.
\end{thebibliography}
\end{document}